\documentclass[11pt]{amsart}

\usepackage{amscd}
\usepackage{amsmath, amssymb}
\usepackage{amsthm}
\usepackage{amsfonts}
\usepackage{enumitem}
\usepackage{verbatim}
\newcommand{\ddt}{\frac{\partial}{\partial t}}
\newcommand{\ddbar}{\sqrt{-1} \partial \overline{\partial}}
\newcommand{\Ric}{\mathrm{Ric}}
\newcommand{\ov}[1]{\overline{#1}}
\newcommand{\tr}[2]{\mathrm{tr}_{#1}{#2}}
\newcommand{\omegahatt}{\hat{\omega}_t}
\newcommand{\omegahatinf}{\hat{\omega}_\infty}
\newcommand{\tphi}{\tilde{\varphi}}

\newcommand{\covaro}{( \nabla_0 )}
\newcommand{\go}{( g_0)}
\newcommand{\To}{(T_0)}

\title{The Chern-Ricci flow on smooth minimal models of general type}
\author{Matthew Gill}
\thanks{Supported by NSF RTG grant DMS-0838703.}

\begin{document}
\newcounter{remark}
\newcounter{theor}
\setcounter{remark}{0}
\setcounter{theor}{1}
\newtheorem{claim}{Claim}
\newtheorem{theorem}{Theorem}[section]
\newtheorem{lemma}[theorem]{Lemma}
\newtheorem{corollary}[theorem]{Corollary}
\newtheorem{proposition}[theorem]{Proposition}
\newtheorem{question}{question}[section]
\newtheorem{defn}{Definition}[theor]

\begin{abstract}
We show that on a smooth Hermitian minimal model of general type the Chern-Ricci flow converges to a closed positive current on $M$. Moreover, the flow converges smoothly to a K\"ahler-Einstein metric on compact sets away from the null locus of $K_M$. This generalizes work of Tsuji and Tian-Zhang to Hermitian manifolds, providing further evidence that the Chern-Ricci flow is a natural generalization of the K\"ahler-Ricci flow.
\end{abstract}

\maketitle

\section{Introduction}

Let $(M,g_0)$ be a complex manifold of dimension $n$ with a Hermitian metric $g_0$. We define a real $(1,1)$ form $\omega_0 = \sqrt{-1} (g_0)_{i\bar{j}} dz^i \wedge dz^{\overline{j}}$ on $M$. The normalized Chern-Ricci flow is

\begin{equation}\label{ncrf}
\ddt \omega = - \Ric(\omega) - \omega, \ \ \ \omega|_{t=0} = \omega_0
\end{equation} 
where $\Ric(\omega) := -\ddbar \log \det g$ is the Chern-Ricci form of $\omega$. When the initial metric $\omega_0$ is K\"ahler ($d\omega_0 = 0$), then \eqref{ncrf} is the normalized K\"ahler-Ricci flow. Another flow of Hermitian metrics, the pluriclosed flow, has been considered by Streets-Tian \cite{StT1, StT2, StT3} (see also Liu-Yang \cite{LY}). The unnormalized Chern-Ricci flow was introduced in \cite{G}. The overall hope is that the Chern-Ricci flow will be useful in the classification of complex surfaces much like the Ricci flow in real dimension three \cite{H1, H2, H3, P1, P2, P3}.

Recently, the Chern-Ricci flow has been shown to have many properties in common with the K\"ahler-Ricci flow, especially in the case of complex surfaces. When the first Bott-Chern class is zero, the flow was shown to exist for all time and converge smoothly to a Chern-Ricci flat metric \cite{G} using estimates for the elliptic Monge-Amp\`ere equation \cite{Ch, GL, TW1}. This generalized the K\"ahler case considered by Cao \cite{Cao}, whose proof made use of the estimates of Yau \cite{Yau}. The work of Tosatti-Weinkove \cite{TW2, TW3} contains several explicit examples of the Chern-Ricci flow and many results generalizing those of the K\"ahler-Ricci flow. In particular, that the flow exists on some maximal time interval that depends on the Bott-Chern class of the initial metric. If the first Chern class of the manifold is negative, then the flow starting with any Hermitian metric converges smoothly to a K\"ahler-Einstein metric. On complex surfaces with an initial Gauduchon metric, the flow exists either for all time or until the volume or a curve of negative self-intersection tends to zero. Starting with an elliptic bundle over a Riemann surface of genus greater than one, the Chern-Ricci flow converges exponentially fast to a K\"ahler-Einstein metric on the base \cite{TWY}. Local Calabi and curvature estimates are also known for the flow \cite{ShW2}. Analogous results for the K\"ahler-Ricci flow can be found in \cite{Cao, FIK, FZ, G2, ShW1, ST1, ST2, SW1, SWnotes,TZ}.

If the first Bott-Chern class of the canonical bundle $K_M$ is nef, we say that $M$ is a \emph{minimal model}. When $M$ is a minimal model, the normalized Chern-Ricci flow has a smooth solution for all time \cite{TWY}. Additionally, if $K_M$ is a big line bundle, we say that $M$ is of \emph{general type}. The null locus of $K_M$, $\mathrm{Null}(K_M)$, is the union over all positive dimensional irreducible analytic subvarieties $V \subset M$ of dimension $k$ where
\begin{equation*}
\int_V\left( c_1(K_M)\right)^k = 0.
\end{equation*}
We assume that $M$ is a Hermitian smooth minimal model of general type and prove the following theorem:

\begin{theorem}\label{maintheorem}
Let $(M,\omega_0)$ be a smooth Hermitian minimal model of general type of dimension $n$ with Hermitian metric $\omega_0$. Then the normalized Chern-Ricci flow \eqref{ncrf} has a smooth solution for all time and there exists a closed positive current $\omega_{KE}$ on $M$ such that $\omega(t)$ converges $\omega_{KE}$ as currents as $t \to \infty$. 

Moreover, letting $E = \mathrm{Null}(K_M)$, $\omega(t)$ converges in $C^\infty_{loc}(M\setminus E)$ to a K\"ahler-Einstein metric $\omega_{KE}$ away from $E$ satisfying 
\begin{equation*}
\Ric(\omega_{KE}) = - \omega_{KE}.
\end{equation*}
The null locus of $K_M$ is the smallest possible choice for $E$. 
\end{theorem}

As an immediate corollary, we see that every smooth Hermitian minimal model of general type has a closed positive current which is a K\"ahler-Einstein metric away from the null locus of $K_M$. Additionally, $\omega_{KE}$ is unique in a sense that will be defined at the end of the introduction. The statement that the null locus of $K_M$ is the smallest choice for $E$ follows from the recent work of Collins-Tosatti \cite{CT}.

 In dimension $n = 2$, $M$ is projective. This is not true in general for $n > 2$. If $M$ is K\"ahler and we start the flow with a K\"ahler metric this is the result of Tsuji \cite{Ts} and Tian-Zhang \cite{TZ}. The difference in the above theorem is that $M$ need not be K\"ahler. If $M$ is K\"ahler and the initial metric is not K\"ahler, the main theorem implies that the Chern-Ricci flow still tends to the same limit as in the work of Tsuji and Tian-Zhang. This suggests that the Chern-Ricci flow is a natural object of study. 

We now provide a brief outline of the proof. As in the K\"ahler case, we reduce to a complex parabolic Monge-Amp\`ere equation
\begin{equation*}
\ddt \varphi = \log \frac{(\hat{\omega}_t + \ddbar \varphi)^n}{\Omega} - \varphi, \ \ \ \varphi |_{t=0} = 0, \ \ \ \hat{\omega}_t + \ddbar \varphi > 0,
\end{equation*}
where $\omegahatt$ is a reference metric and $\Omega$ is a volume form. Following the K\"ahler case we have uniform upper bounds for $\varphi, \dot\varphi$ and $\omega^n$. Applying a trick from Collins-Tosatti \cite{CT}, we find a closed positive current
\begin{equation*}
T = \omegahatinf + \ddbar \psi \in -c_1^{BC}(M)
\end{equation*}
with $T \geq c_0 \omega$ as currents on $M$. Here $\psi$ is is an upper-semi continuous function in $L^1(M)$ with $\sup_M \psi = 0$ and is smooth away from $E = \mathrm{Null}(K_M)$. We find uniform bounds for $\varphi, \dot\varphi$ and $\omega^n$ in terms of $\psi$. Letting $\tphi = \varphi - \psi$ and using the Phong-Sturm term \cite{PS}
\begin{equation*}
\frac{1}{\tphi + C_0}
\end{equation*}
we define the quantity 
\begin{equation*}
Q = \log \tr {\omega_0} \omega - A \tphi + \frac{1}{\tphi + C_0}
\end{equation*}
as in \cite{TW2}. Using the maximum principle we obtain the estimate
\begin{equation*}
\tr {\omega_0} \omega \leq \frac{C'}{e^{C\psi}}.
\end{equation*}
Applying the higher order estimates from \cite{G} and the bounds for $\dot{\varphi}$ on compact subsets of $M\setminus E$, we prove smooth convergence. We also have the following uniqueness result which follows immediately:
\begin{theorem}\label{uniquenessthm}
$\omega_{KE}$ is the unique closed, positive current on $M$ smooth on $M \setminus E$ satisfying
\begin{enumerate}[label=\emph{(\roman*)}]
\item $\omega_{KE} = - \Ric (\omega_{KE})$ on $M\setminus E$ and
\item $\frac{1}{C_\varepsilon} e^{\varepsilon \psi} \Omega \leq \omega_{KE}^n \leq C \Omega$ for all $\varepsilon \in (0,1]$ on $M \setminus E$.
\end{enumerate}
This is independent of choice of $\Omega$. 
\end{theorem}

\section{Preliminaries}

In this section we will review some of the notation used in the proof of the main theorem. For a more detailed discussion, we refer the reader to \cite{TW2}. Every Hermitian metric $g$ has an associated $(1,1)$ form 
\begin{equation*}
\omega = \sqrt{-1} g_{i\ov j} dz^i \wedge dz^{\ov j}.
\end{equation*}
The metric also has a Chern connection $\nabla$ with Christoffel symbols 
\begin{equation*}
\Gamma_{ij}^k = g^{\ov l k} \partial_{i} g_{j\ov{l}}. 
\end{equation*}
The torsion of the metric is the tensor
\begin{equation*}
T_{ij}^k = \Gamma_{ij}^k - \Gamma_{ji}^k.
\end{equation*}
If $g$ is a K\"ahler metric, then the torsion of $g$ is zero. The Chern curvature of $g$ is
\begin{equation*}
{R_{k\ov l i}}^p = -\partial_{\ov l} \Gamma_{ki}^p
\end{equation*}
and it obeys the usual commutation identities for curvature. For example,
\begin{equation*}
[\nabla_k, \nabla_{\ov l}] X^i = {R_{k \ov l j}}^i X^j.
\end{equation*}
The Chern-Ricci curvature of $g$ is
\begin{equation*}
R_{k\ov l} = g^{\ov j i} R_{k \ov l i \ov j} = - \partial_k \partial_{\ov l} \log \det g
\end{equation*}
with associated Chern-Ricci form
\begin{equation*}
\Ric (\omega) = \sqrt{-1} R_{k\ov l} dz^k \wedge dz^{\ov l}.
\end{equation*}

\section{Estimates}

First, we need to choose an appropriate reference metric. Since $-c_1^{BC}(M)$ is nef and $K_M$ is big, M is Moishezon, and we can apply \cite{KMM} to find a non-negative $(1,1)$ form $\hat{\omega}_\infty$ such that 
\begin{equation*}
[\hat{\omega}_\infty] = - c_1^{BC}(M).
\end{equation*}
Additionally, there exists a smooth volume form $\Omega$ such that
\begin{equation*}
\ddbar \log \Omega = \hat{\omega}_\infty, \ \ \ \int_M \Omega = \int_M \omega_0^n.
\end{equation*}
Define a family of reference metrics 
\begin{equation*}
\hat{\omega}_t = e^{-t} \omega_0 + (1-e^{-t})\hat{\omega}_\infty.
\end{equation*}
If $\varphi$ solves
\begin{equation}\label{pma}
\ddt \varphi = \log \frac{(\hat{\omega}_t + \ddbar \varphi)^n}{\Omega} - \varphi, \ \ \ \varphi |_{t=0} = 0, \ \ \ \hat{\omega}_t + \ddbar \varphi > 0
\end{equation}
then $\omega = \omegahatt + \ddbar \varphi$ solves the normalized Chern-Ricci flow \eqref{ncrf}.

We require some standard estimates for $\varphi$ that follow as in the K\"ahler case \cite{Ts, TZ}. For a recent exposition of this result, see \cite{SWnotes}.

\begin{lemma}\label{phibddabove}
There exists a uniform $C$ such that on $M \times [0,\infty)$,
\begin{enumerate}[label=\emph{(\roman*)}]
\item $\varphi \leq C$
\item $\dot{\varphi} \leq Cte^{-t}$ when $t \geq t_1$ for some $t_1 > 0$. In particular, $\dot{\varphi} \leq C$
\item $\omega^n \leq C\Omega$.
\end{enumerate}
\end{lemma}

We need a version of Tsuji's trick \cite{Ts} that will apply in this non-K\"ahler setting. The new trick comes from the work of Collins--Tosatti \cite{CT} and a theorem of Demailly \cite{D} and Demailly-P\u{a}un \cite{DP}:

Since $K_M$ is big there exists a K\"ahler current
\begin{equation}\label{defofT}
T = \omegahatinf + \ddbar \psi \geq c_0 \omega_0
\end{equation}
for some $c_0 > 0$ as currents on $M$ where $\psi$ is an upper-semi continuous function in $L^1(M)$. Moreover, $\psi$ can be chosen to be smooth away from a closed analytic subvariety
\begin{equation*}
E = \{ \psi = -\infty \}.
\end{equation*}
By adding a constant, we can assume that $\sup_M \psi = 0$. Since $M$ is Moishezon, it is in Fujiki's class $\mathcal{C}$  (M is bimeromorphic to a compact K\"ahler manifold) \cite{F}. Using this fact, the main theorem of Collins-Tosatti implies that we can take
\begin{equation*}
E = \mathrm{Null}(K_M)
\end{equation*} 
and that this is the smallest possible choice for $E$ \cite{CT}.

From the definition of $T$ and $\psi$ we have the following useful facts.

\begin{lemma}\label{psifacts}
There exists a uniform $C > 0$ such that
\begin{enumerate}[label=\emph{(\roman*)}]
\item $\ddbar \psi \geq - C \omega_0$ as currents on $M$ and
\item $\omegahatinf + \varepsilon \ddbar \psi \geq \varepsilon c_0 \omega_0$ as currents on $M$ for all $\varepsilon \in (0,1]$.
\end{enumerate}
\end{lemma}

We can find lower bounds for $\varphi, \tphi$ and $\omega^n$ in terms of $\psi$ and $\varepsilon.$
\begin{lemma}\label{tphibelow}
There exists a uniform constant $C_\varepsilon$ depending on $\varepsilon$ such that on $M \times [0,\infty)$,
\begin{enumerate}[label=\emph{(\roman*)}]
\item $\varphi \geq \varepsilon \psi -C_\varepsilon$
\item $\dot{\varphi} \geq \varepsilon \psi - C_\varepsilon$
\item $\omega^n \geq \frac{1}{C_\varepsilon} e^{\varepsilon \psi} \Omega$.
\end{enumerate}
\end{lemma} 

\begin{proof}
Define
\begin{equation*}
Q = \dot{\varphi} + \varphi - \varepsilon \psi = \log \frac{\omega^n}{e^{\varepsilon \psi} \Omega}.
\end{equation*}
If we can find a uniform lower bound for $Q$ we immediately prove (iii). (i) and (ii) then follow from Lemma \ref{phibddabove}. Computing the evolution equation for Q,
\begin{align*}
\left(\ddt - \Delta \right) \varphi &= \dot{\varphi} - n + \tr {\omega} {\omegahatt} \\
\left(\ddt - \Delta \right) \dot\varphi &= \tr {\omega} {\omegahatinf - \omegahatt} - \dot\varphi.
\end{align*}
Adding these,
\begin{align*}
\left( \ddt - \Delta \right) Q &= \tr \omega \omegahatinf - n + \tr \omega {\varepsilon \ddbar\psi} \\
&= \tr \omega {\left(\omegahatinf + \varepsilon \ddbar \psi \right)} - n \\
&\geq \varepsilon c_0 \tr \omega \omega_0 - n.
\end{align*}
Since $Q \to \infty$ as $x \to E$, $Q$ achieves a spatial minimum for each fixed time $t_0$. If $Q$ attains a minimum at the point $(x_0,t_0)$ in $M \setminus E$ with $t_0 > 0$, at that point
\begin{equation*}
\tr {\omega(x_0,t_0)} {\omega_0(x_0,t_0)} \leq \frac{n}{\varepsilon c_0}.
\end{equation*}
Applying the geometric-arithmetic mean inequality,
\begin{equation*}
\left( \frac{\omega_0^n(x_0,t_0)}{\omega^n(x_0,t_0)}\right)^{1/n} \leq \frac{\tr {\omega(x_0,t_0)} {\omega_0(x_0,t_0)}}{n} \leq \frac{1}{\varepsilon c_0}.
\end{equation*}
This gives a uniform lower bound for $Q$ since
\begin{equation*}
Q(x_0,t_0) = \log \frac{\omega^n(x_0,t_0)}{e^{\varepsilon \psi (x_0,t_0)} \Omega(x_0,t_0)} \geq  \log \frac{\varepsilon c_0 \omega_0^n(x_0,t_0)}{e^{\varepsilon \psi (x_0,t_0)} \Omega(x_0,t_0)} \geq - C_\varepsilon.
\end{equation*}
\end{proof}

We define a family of positive $(1,1)$-currents which will be useful in bounding $\tr {\omega_0} \omega$. Let
\begin{equation*}
S_t = \omegahatt + \ddbar \psi = e^{-t} \omega_0 + \left(1-e^{-t}\right) \omegahatinf + \ddbar \psi.
\end{equation*}
\begin{lemma}\label{Sbelow}
There exists $T_0 > 0$ such that for all $t \geq T_0$
\begin{equation}\label{Sbeloweqn}
S_t \geq \frac{c_0}{2} \omega_0
\end{equation}
as currents on $M$.
\end{lemma}
\begin{proof}
Choose $T_0$ sufficiently large so that for $t \geq T_0$
\begin{equation*}
e^{-t} \left(\omega_0 - \omegahatinf \right) \geq - \frac{c_0}{2} \omega_0.
\end{equation*}
Then
\begin{equation*}
S_t = \omegahatinf + \ddbar \psi + e^{-t} \left(\omega_0 - \omegahatinf \right) \geq \frac{c_0}{2} \omega_0 
\end{equation*}
as currents on $M$.
\end{proof}

Now we can bound $\tr {\omega_0} \omega$ using a trick from Phong-Sturm \cite{PS}.

\begin{lemma}\label{tracebound}
There exists uniform $C$ and $C'$ such that on $M \times [0,\infty)$,
\begin{equation}\label{eqntracebound}
\tr {\omega_0} \omega \leq \frac{C'}{e^{C\psi}}.
\end{equation}
Moreover, there exists uniform $C''$ such that on $M \times [0,\infty)$,
\begin{equation}\label{unifequiv}
\frac{e^{C\psi}}{C''}\omega_0 \leq \omega \leq \frac{C''}{e^{C\psi}} \omega_0.
\end{equation}
\end{lemma}
\begin{proof}
We begin by calculating the evolution equation for $\log \tr {\omega_0} \omega$ following a method similar to \cite{TW2}.
\begin{align}\label{bigevoeqn}
\left(\ddt - \Delta \right)& \log \tr {\omega_0} \omega \nonumber \\
 = &\frac{1}{\tr {\omega_0} \omega} \Bigg( \Big[ -g^{\overline{j}p} g^{\overline{q}i} (g_0)^{\overline{l}k} (\nabla_0)_k g_{i\overline{j}} (\nabla_0)_{\ov l} g_{p\overline{q}} \nonumber \\
 & \ + \frac{1}{\tr {\omega_0} \omega} g^{\overline{l}k} \covaro_k (\tr {\omega_0} \omega) \covaro_{\ov l} (\tr {\omega_0} \omega) \nonumber \\ 
 & \ - 2 \mathrm{Re}\left( g^{\ov j i} (g_0)^{\ov l k} (T_0)^p_{ki} (\nabla_0)_{\ov l} g_{p\ov j} \right)  - g^{\ov j i} (g_0)^{\ov l k} g_{p\ov q} (T_0)^p_{ik} \overline{(T_0)^q_{jl}} \Big] \nonumber\\
 & + \Big[ g^{\ov j i} (g_0)^{\ov l k} g_{k \ov q} \left( (\nabla_0)_i \overline{(T_0)^q_{jl}} - (R_0)_{i\ov l p \ov j} (g_0)^{\ov q p} \right)  \Big]   \nonumber   \\
 & - \Big[ g^{\ov j i} \go^{\ov l k} e^{-t} \Big( \covaro_i \left( \ov{\To_{jl}^p} \go_{k\ov p} \right) + \covaro_{\ov l} \left( \To_{ik}^p \go_{p \ov j} \right) \Big) \nonumber \\
 & \ -  g^{\ov j i} \go^{\ov l k} e^{-t} \ov{ \To_{jl}^q} \ov{\To_{ik}^p} \go_{p\ov q} \Big] - \tr {\omega_0} \omega \Bigg).
\end{align}
There two differences between this equation and the one in \cite{TW2}.  The third term in square brackets has a factor of $e^{-t}$ since
\begin{equation}\label{theeminust}
\ov \partial \omega = \ov \partial \omegahatt = e^{-t} \ov \partial \omega_0.
\end{equation}
Also, because we are considering the normalized Chern-Ricci flow we have the final $-\tr {\omega_0} \omega$ term.

Let $(I)$ denote the first term in square brackets in \eqref{bigevoeqn}, $(II)$ the second, and $(III)$ the third, all including the $\frac{1}{\tr {\omega_0} \omega}$ out front. Using the estimates from \cite{TW2} Proposition 3.1, 
\begin{align}\label{estI}
(I) & \leq  \frac{2e^{-t}}{(\tr {\omega_0} \omega)^2} \mathrm{Re} \left( g^{\ov l k} (T_0)^p_{kp} \partial_{\ov l} \tr {\omega_0} \omega \right)  \nonumber \\
& \leq  \frac{2}{(\tr {\omega_0} \omega)^2} \mathrm{Re} \left( g^{\ov l k} (T_0)^p_{kp} \partial_{\ov l} \tr {\omega_0} \omega \right).
\end{align}
The above estimate differs from the one in \cite{TW2} by a factor of $e^{-t}$ which again comes from equation \eqref{theeminust}. By definition of the respective terms, 
\begin{equation}\label{estII}
(II) \leq C \tr {\omega} {\omega_0},
\end{equation}
and
\begin{equation}\label{estIII}
(III) \leq C e^{-t} \tr {\omega} {\omega_0} \leq  C \tr {\omega} {\omega_0}.
\end{equation}
Combining \eqref{estI}, \eqref{estII} and \eqref{estIII} with \eqref{bigevoeqn},
\begin{equation}\label{evologtr1}
\left(\ddt - \Delta \right) \log \tr {\omega_0} \omega \leq \frac{2}{(\tr {\omega_0} \omega)^2} \mathrm{Re} \left( g^{\ov l k} (T_0)^p_{kp} \partial_{\ov l} \tr {\omega_0} \omega \right) + C \tr \omega {\omega_0}.
\end{equation}
Let 
\begin{equation*}
\tphi = \varphi - \psi.
\end{equation*}
Using the trick from Phong-Sturm \cite{PS}, we consider the quantity
\begin{equation}\label{logtrQ}
Q = \log \tr {\omega_0} \omega - A \tphi + \frac{1}{\tphi + C_0}.
\end{equation}
Here $A$ is a large constant to be determined later and $C_0$ is large enough so that $\tphi + C_0 \geq 1$ which exists by Lemma \ref{tphibelow}. This choice is made so that
\begin{equation*}
0 < \frac{1}{\tphi + C_0} \leq 1.
\end{equation*}
Fix a time $T' > T_0$ where $T_0$ is as in Lemma \ref{Sbelow}. Since $Q \to -\infty$ as $x \to E$, $Q$ achieves a maximum at some point $(x_0, t_0) \in (M \setminus E ) \times [0,T'].$ If $0 \leq t_0 \leq T_0$, then $Q$ clearly has a uniform upper bound on $M \times [0, T']$. It remains to show that $Q$ is uniformly bounded above if $t_0 > T_0$. 

We compute the parts of the evolution equation for $Q$ separately.
\begin{align}\label{evolutiontphi}
\left(\ddt - \Delta \right) \tphi & =  \dot{\varphi} - \tr \omega{\left( \ddbar \varphi -  \ddbar\psi \right)} \\
\nonumber & =  \dot\varphi -  \tr \omega{\left( \omega - \omegahatt -  \ddbar \psi \right)} \\
\nonumber & = \dot\varphi - n  + \tr \omega S_t.
\end{align}
Using the previous calculation that showed $\Delta \tphi = n - \tr \omega S_t$,
\begin{align}\label{evofracterm}
\left(\ddt - \Delta \right) \frac{1}{\tphi + C_0} & = - \frac{\dot{\varphi}}{(\tphi + C_0)^2} + \frac{\Delta \tphi}{(\tphi + C_0)^2} - \frac{2|\partial \tphi |_g^2}{(\tphi + C_0)^3} \\
\nonumber &= - \frac{\dot{\varphi}}{(\tphi + C_0)^2} + \frac{n - \tr \omega S_t}{(\tphi + C_0)^2} - \frac{2|\partial \tphi |_g^2}{(\tphi + C_0)^3}.
\end{align}
Combining \eqref{logtrQ}, \eqref{evolutiontphi} and \eqref{evofracterm},
\begin{align}\label{evologtrQ1}
\left(\ddt - \Delta \right) Q = & \left(\ddt - \Delta \right) \log \tr {\omega_0} \omega - \left(A+\frac{1}{(\tphi+C_0)^2}\right) \dot\varphi \\ 
\nonumber & +\left(A+\frac{1}{(\tphi+C_0)^2}\right)\left(n - \tr \omega S_t \right) - \frac{2|\partial \tphi |_g^2}{(\tphi + C_0)^3}.
\end{align}
At the maximum of $Q$, $(x_0,t_0)$, 
\begin{equation}\label{dQ0}
0 = \partial_{\ov l} Q = \frac{1}{\tr {\omega_0} \omega} \partial_{\ov l} \tr {\omega_0} \omega - A \partial_{\ov l} \tphi - \frac{\partial_{\ov l} \tphi}{(\tphi + C_0)^2}.
\end{equation}
Substituting \eqref{dQ0} in to the first term in \eqref{evologtr1} at $(x_0,t_0)$,
\begin{align}\label{badtermbound}
\bigg\vert \frac{2}{(\tr {\omega_0} \omega)^2} \mathrm{Re} & \left( g^{\ov l k} (T_0)^p_{kp} \partial_{\ov l} \tr {\omega_0} \omega \right) \bigg\vert \\
\nonumber & =  \left\vert \frac{2}{(\tr {\omega_0} \omega)^2} \mathrm{Re} \left( g^{\ov l k} (T_0)^p_{kp} \left( A + \frac{1}{(\tphi+C_0)^2}\right)\partial_{\ov l} \tphi \right) \right\vert \\
\nonumber & \leq \frac{ |\partial \tphi|_g^2}{(\tphi+C_0)^3} + CA^2 (\tphi+C_0)^3 \frac{ \tr \omega {\omega_0}}{( \tr {\omega_0} \omega)^2}.
\end{align}
Now we break this in to two cases. If $(\tr {\omega_0} \omega)^2 \leq A^2(\tphi + C_0)^3$ at $(x_0,t_0)$, then
\begin{equation*}
Q \leq \log A + \frac{3}{2} \log (\tphi + C_0) - A \tphi + \frac{1}{\tphi + C_0} \leq C
\end{equation*}
where $C$ is some constant depending on $A$ since $\tphi$ is bounded below and the function $x \mapsto \frac{3}{2}\log (x + C_0) - Ax$ is bounded above. 

If instead $(\tr {\omega_0} \omega)^2 \geq A^2(\tphi + C_0)^3$ at $(x_0,t_0)$, substituting \eqref{badtermbound} into \eqref{evologtrQ1},
\begin{align*}
\left(\ddt - \Delta \right) Q \leq & \ \frac{ |\partial \tphi|_g^2}{(\tphi+C_0)^3} + C \tr \omega {\omega_0} - \left(A+\frac{1}{(\tphi+C_0)^2}\right) \dot\varphi \\ 
\nonumber & +\left(A+\frac{1}{(\tphi+C_0)^2}\right)\left(n - \tr \omega S_t \right) - \frac{2|\partial \tphi |_g^2}{(\tphi + C_0)^3}.\\
\nonumber \leq & \ C \tr \omega {\omega_0} + \left(A+1\right) |\dot\varphi|  +\left(A+1\right)\left(n - \tr \omega S_t \right).
\end{align*}
Using Lemma \ref{Sbelow}, we can choose $A$ large enough so that $(A+1) S_t \geq (C+1)\omega_0$. At $(x_0,t_0)$,
\begin{equation*}
0 \leq \left(\ddt - \Delta \right) Q \leq - \tr \omega {\omega_0} + C \left\vert \log \frac{\Omega}{\omega^n} \right\vert + C.
\end{equation*}
At the maximum of $Q$,
\begin{equation*}
\tr {\omega_0} \omega \leq \frac{1}{(n-1)!} (\tr \omega {\omega_0})^{n-1} \frac{\omega^n}{\omega_0^n} \leq C \frac{\omega^n}{\Omega} \left\vert \log \frac{\Omega}{\omega^n}\right\vert^{n-1} + C \leq C
\end{equation*}
and since $\omega^n \leq C \Omega$ and the function $x \mapsto x | \log x |^{n-1}$ is bounded for small $x>0$ we obtain a uniform upper bound for $Q$. In either of the cases, $Q$ is uniformly bounded above, so using Lemma \ref{phibddabove}
\begin{equation*}
\log \tr {\omega_0} \omega \leq C + C \tphi \leq C - C \psi. 
\end{equation*}
Exponentiating gives \eqref{eqntracebound} and \eqref{unifequiv} follows.
\end{proof}

Using these lower order estimates with the higher order estimates in \cite{G} on compact subsets of $M \setminus E$ we obtain uniform $C^\infty_{loc}(M \setminus E)$ estimates for $\varphi$. 

\section{Convergence and uniqueness}

We now complete the proof of Theorem \ref{maintheorem} by showing that $\omega$ converges to a K\"ahler-Einstein metric on $M \setminus E$ in $C^\infty_{loc}(M \setminus E)$. 
\begin{proof}
The quantity
\begin{equation*}
Q = \varphi + Cte^{-t}
\end{equation*}
where $C$ is the constant in Lemma \ref{phibddabove} is uniformly bounded below on compact subsets of $M \setminus E$ by Lemma \ref{tphibelow}. By Lemma \ref{phibddabove} (ii),
\begin{equation*}
\ddt Q = \dot\varphi - Cte^{-t} \leq 0
\end{equation*}
so $\varphi$ converges pointwise to a function $\varphi_\infty$ at $t \to \infty$ on $M \setminus E$. Using the estimates from the previous section, we have convergence in $C^\infty_{loc} (M\setminus E)$. Since $\varphi$ converges as $t \to \infty$, $\dot\varphi \to 0$ similarly in $C^\infty_{loc} (M\setminus E)$.

The above convergence for $\varphi$ and $\dot\varphi$ implies that
\begin{equation*}
\omega \to \omega_\infty := \omegahatinf + \ddbar \varphi_\infty
\end{equation*}
and $\ddt \omega \to 0$ as $t \to \infty$ in $C^\infty_{loc} (M\setminus E)$. Taking $t \to \infty$ in the normalized Chern-Ricci flow \eqref{ncrf},
\begin{equation*}
\Ric(\omega_\infty) = -\omega_\infty
\end{equation*}
on $M \setminus E$. Since $\Ric(\omega_\infty)$ is closed, $\omega_\infty$ is a K\"ahler-Einstein metric on $M \setminus E$. Moreover, applying weak compactness of currents we can extend $\omega_\infty$ to a closed, positive current on $X$ and then $\omega \to \omega_\infty$ as currents on $X$.
\end{proof}

We now show that $\omega_\infty$ is unique in the sense of Theorem \ref{uniquenessthm}:
\begin{proof}
The proof of this result is similar to the K\"ahler case \cite{Ts, TZ} (see also \cite{SWnotes}), but we provide the proof for the sake of completeness. Let $\omega_\infty$ and $\tilde{\omega}_\infty$ be two closed positive currents satisfying (i) and (ii). Define
\begin{equation*}
\theta = \log \frac{\omega_\infty^n}{\Omega}, \ \ \ \tilde\theta = \log \frac{\tilde{\omega}_\infty^n}{\Omega}.
\end{equation*}
Taking $\ddbar$,
\begin{equation*}
\ddbar \theta = - \Ric(\omega_\infty) - \omegahatinf
\end{equation*}
and so
\begin{equation*}
\omega_\infty = \omegahatinf + \ddbar \theta.
\end{equation*}
Similarly
\begin{equation*}
\tilde\omega_\infty = \omegahatinf + \ddbar \tilde\theta.
\end{equation*}
Define the quantity
\begin{equation*}
Q = \theta - (1-\delta) \tilde\theta - \delta \varepsilon \psi
\end{equation*}
for some $0 < \delta < 1$. $Q$ is bounded below and $Q \to \infty$ as $x \to E$ so $Q$ attains a minimum at a point $x_0$ in $M \setminus E$. At $x_0$,
\begin{align*}
\theta - \tilde\theta &= \log \frac{\omega_\infty^n}{\tilde\omega_\infty^n} \\
&= \log \frac{\left(\omegahatinf + (1-\delta) \ddbar \tilde\theta + \delta\varepsilon \ddbar \psi + \ddbar Q \right)^n}{\tilde\omega_\infty^n} \\
&= \log \frac{\left( (1-\delta)\left(\omegahatinf + \ddbar \tilde\theta \right) + \delta \left( \omegahatinf + \varepsilon \ddbar \psi \right) + \ddbar Q  \right)^n}{\tilde\omega_\infty^n} \\
& \geq \log \frac{ (1-\delta)^n \tilde\omega_\infty^n}{ \tilde\omega_\infty^n} = n \log (1-\delta).
\end{align*}
By Lemma \ref{tphibelow}, $\delta \tilde \theta (x_0) - \delta \varepsilon \psi (x_0) \geq \delta C_\varepsilon$. Then
\begin{equation*}
Q(x_0) = \theta(x_0) -\tilde\theta(x_0) + \delta \tilde \theta (x_0) - \delta \varepsilon \psi (x_0) \geq n \log (1-\delta) - \delta C_\varepsilon.
\end{equation*}
Choosing $\delta$ sufficiently small so that $ n \log (1-\delta) > -\varepsilon / 2$ and $ \delta C_\varepsilon < \varepsilon /2$,
\begin{equation*}
Q(x_0) \geq - \varepsilon.
\end{equation*}
Since $Q$ achieves its minimum at $x_0$,
\begin{equation*}
\theta \geq (1-\delta)\tilde\theta + \delta \varepsilon \psi - \varepsilon.
\end{equation*}
Taking $\delta \to 0$ and $\varepsilon \to 0$, 
\begin{equation*}
\theta \geq \tilde \theta.
\end{equation*}
Similarly, $\theta \leq \tilde{\theta}$.
\end{proof}

\section{Acknowledgments}

The author would like to thank Ben Weinkove and Valentino Tosatti for several helpful discussions and suggestions.

\bigskip
\noindent
Department of Mathematics, University of California, Berkeley, 970 Evans Hall \#3840, Berkeley, CA 94720-3840 USA

\end{document}